\documentclass[12pt]{article}

\usepackage{amssymb}
\usepackage{amsthm}
\usepackage{amsmath}
\usepackage{graphicx}
\usepackage{fullpage}
\usepackage{color}
 \numberwithin{equation}{section}

\theoremstyle{plain}
\newtheorem{thm}{Theorem}[section]
\newtheorem{cor}[thm]{Corollary}

\newtheorem{lem}[thm]{Lemma}
\newtheorem{prop}[thm]{Proposition}

\theoremstyle{definition}

\theoremstyle{remark}

\newtheorem{rem}[thm]{Remark}

\newcommand{\R}{\mathbb{R}}

\newcommand{\E}{\mathbb{E}}

\newcommand{\I}{\infty}

\newcommand{\tr}{\text{tr}}
\newcommand{\bp}{\begin{proof}[\ensuremath{\mathbf{Proof}}]}
\newcommand{\bs}{\begin{proof}[\ensuremath{\mathbf{Solution}}]}
\newcommand{\ep}{\end{proof}}

\begin{document}


\title{Analysis of Hamilton-Jacobi-Bellman equations \\ arising in stochastic singular control}

\author{Ryan Hynd\thanks{This material is based upon work supported by the National Science Foundation under Grant No. DMS-1004733.}\\
Courant Institute of Mathematical Sciences\\
New York University\\
251 Mercer Street\\
New York, NY 10012-1185 USA}

\maketitle

\begin{abstract}
We study the partial differential equation
$$
\max\left\{Lu - f, H(Du)\right\}=0
$$
where $u$ is the unknown function, $L$ is a second-order elliptic operator, $f$ is a given smooth function and $H$ is a convex function. 
This is a model equation for Hamilton-Jacobi-Bellman equations arising in stochastic singular control.  We establish the existence of a unique viscosity solution of the Dirichlet problem that has a H\"{o}lder continuous gradient.  We also show that if $H$ is {\it uniformly} convex, the gradient of this solution is Lipschitz continuous.
\end{abstract}

\tableofcontents

\section{Introduction}
In this paper, we consider PDE associated with a general class of stochastic singular
control problems. This is a class of nonlinear, second-order PDE that each have a free boundary determined by a convex gradient constraint. 
Using PDE methods, we show that the Dirichlet problem has a unique solution and derive some regularity properties of the gradient of this solution. 
Namely, we establish that the gradient is H\"{o}lder continuous and also that if $H$ is uniformly convex, the gradient is Lipschitz continuous. Finally, 
we give a brief discussion of how this type of equation arises in control theory and show how our regularity results apply to the motivating control problems.

\par  The PDE we focus on is

\begin{equation}\label{mainPDE}
\begin{cases}
\max\{Lu-f,H(Du)\}=0, \; & x\in O\\
\hspace{1.46in} u=0, \; & x\in \partial O
\end{cases},
\end{equation}
where $O\subset \R^n$ is open and bounded with smooth boundary $\partial O$ and $f$ is a smooth, non-negative function on $\overline{O}. $ We assume that $L$ is the linear differential operator

\begin{equation}
L\psi(x):= - a(x)\cdot D^2\psi + b(x)\cdot D\psi + c(x)\psi,\quad \psi\in C^2(\overline{O})\nonumber
\end{equation}
with smooth coefficients $a: \overline{O}\rightarrow {\mathcal S}(n),$ $b:\overline{O}\rightarrow \R^n$ and  $c:\overline{O}\rightarrow \R$. Here $A\cdot B:=\tr A^t B$ and ${\mathcal S}(n)$ denotes the set of $n\times n$ symmetric matrices with real entries.  We shall further assume that $L$ is {\it (uniformly) elliptic}: 

\begin{equation}\label{aElliptic}
a(x)\xi\cdot \xi \ge \gamma |\xi|^2, \quad\text{for all} \quad x\in \overline{O}, \xi\in\R^n 
\end{equation}
for some $\gamma>0.$  The final assumption on $L$ that we will make is 

\begin{equation}
c(x)\ge \delta, \quad x\in\overline{O} \nonumber
\end{equation}
where $\delta$ is a positive constant.

\par Our central result is 
\begin{thm}\label{mainTheorem}
$(i)$ Assume that $H:\R^n\rightarrow \R$ satisfies 
 
\begin{equation}
\begin{cases}\label{Hcond}
\text{$H$ is convex}\\
H(0)<0 
\end{cases}.
\end{equation}
Then there is a unique viscosity solution 
$$
u\in C^{1,\alpha}_{\text{loc}}(O)\cap C^{0,1}(O)
$$
of \eqref{mainPDE}, for any $\alpha\in (0,1)$.   
\\ $(ii)$ If, in addition to \eqref{Hcond}, $H$ satisfies
\begin{equation}
\begin{cases}\label{Hcond2}
\text{$H$ is  uniformly convex}\\
D^2H\in L^\infty_{\text{loc}}(\R^n) 
\end{cases}
\end{equation}
then $u\in C^{1,1}_{\text{loc}}(O)$.
\end{thm}

\par We employ techniques from the theory of viscosity solutions of scalar non-linear elliptic PDE to prove the existence and uniqueness of solutions of \eqref{mainPDE}.  We also use a penalization 
technique similar to the one introduced by L.C. Evans in \cite{E} and refined by M. Wiegner \cite{W} and H. Ishii et. al. \cite{I} to establish regularity of solutions.  More precisely, we study the {\it penalized} equation
\begin{equation}\label{penalizedPDE}
\begin{cases}
Lu^\epsilon + \beta_\epsilon(H(Du^\epsilon))=f, & x\in O\\
\hspace{1.15in}u^\epsilon = 0, & x\in \partial O
\end{cases},
\end{equation}
where $(\beta_\epsilon)_{\epsilon>0}$ is what we call a {\it penalty} function.  $\beta_\epsilon$ can be thought of as a smoothing of $z\mapsto (z/\epsilon)^+$. We show that this equation has a unique solution $u^\epsilon$ that is bounded in $W^{2,p}_{\text{loc}}(O)$ for all $\epsilon$ positive and small enough and that $u^\epsilon$ converges to $u$ in $C^1_{\text{loc}}(O)$, though a sequence of $\epsilon$ tending to $0$.

\par Our result is novel in the fact that it only assumes {\it convexity} of the gradient constraint function. Some previous regularity results apply to the case of the gradient constraint function
\begin{equation}\label{EvansGradCon}
H(p,x)=|p|-g(x), \quad (p,x)\in \R^n\times \overline{O},
\end{equation}
where $g$ is a smooth, positive function on ${O}$.  While we do not consider gradient constraints that depend on $x$, a close inspection of the methods we employ indicate
they would apply to a large class of these gradient constraints.  For instance, if we assume that $H\in C(\R^n\times\overline{O})$ satisfies
$$
\begin{cases}
p\mapsto H(p,x)\;\; \text{is convex for each $x\in \overline{O}$}\\
H(0,x)<0 \; \; \text{for each $x\in \overline{O}$}\\
x\mapsto H(p,x) \in C^{0,1}(\overline{O})
\end{cases},
$$
then a modification of our methods would establish part $(i)$ of Theorem \ref{mainTheorem} for solutions of equation \eqref{mainPDE} with this $H$.  

\par Our goal was to identify general structural conditions 
on the type of gradient constraints for which penalization methods are successful at yielding regularity results.  We believe that assumptions \eqref{Hcond} and \eqref{Hcond2} accomplish this. In addition, we remark that our main regularity result involves uniform convexity of $H$ while previous results did not explicitly make this assumption.  This is because uniform convexity was built in by their choices of gradient constraints. For example, we can replace the gradient constraint given by $H$ defined in \eqref{EvansGradCon} with $|Du|^2-g(x)^2\le 0$, which is uniformly convex in $Du$. 

\par It also should be noted that $C^2$ regularity of solutions
has been obtained in the case $O=\R^n$ and $H(p)=|p|-1$ by assuming $n=1$ or $2$,  very special structural conditions on $f$ and/or that $L$ has constant coefficients  \cite{SS, SS1}.  The purpose
of this work was to study solutions of \eqref{mainPDE} on a general bounded domain in $\R^n$ and estimate solutions for a {\it general} gradient constraint function and a general elliptic operator.

 \par The organization of this paper is as follows. In section \ref{compprin}, we show that \eqref{mainPDE} has a unique viscosity solution by establishing a comparison principle for sub and supersolutions. Then we pursue the regularity of this solution in section \ref{regularity}
 by studying the penalized equation \eqref{penalizedPDE}.  In the final section, we use our uniform estimates to pass to the limit as $\epsilon\rightarrow 0^+$ and prove Theorem \ref{mainTheorem}.  Before performing our analysis, let us discuss the motivating applications in singular control theory and give a probabilistic interpretation of solutions of \eqref{mainPDE}.

\par {\bf Probabilistic interpretation of solutions.} Assume that $(\Omega, {\mathcal F},\mathbb{P})$ is a probability space equipped with a standard $n$-dimensional Brownian motion $(W(t),t\ge 0)$. A {\it control process} is a pair $(\rho, \xi)$ such that 
$$
\begin{cases}
(\rho(t),\xi(t))\in \R^n\times \R, \; t\ge 0, \\ 
(\rho,\xi)\;\text{is adapted to the filtration generated by $W$}\\
|\rho(t)|=1, t \ge 0, \text{a.s.} \\
\xi(0)=0, \; \text{$t\mapsto \xi(t)$ is non-decreasing and is left continuous with right hand limits a.s.}
\end{cases}.
$$
Now, let $\ell$ be the support function of a nonempty, closed, convex set $K\subset \R^n$. That is, 
\begin{equation}\label{ellapp}
\ell(v):=\sup_{p\in K}v\cdot p, \quad v\in \R^n.
\end{equation}
We consider the stochastic control problem  
\begin{equation}\label{valuefun}
u(x):=\inf_{\rho, \xi}\E^x\int^{\tau}_{0}e^{-\int^{t}_{0}c(X^{\rho,\xi}(s))ds}\left[f(X^{\rho,\xi}(t))dt+\ell(\rho(t))d\xi(t)\right],\quad x\in\overline{O}.
\end{equation}
Here $X^{\rho,\xi}$ satisfies the stochastic differential equation (SDE)

\begin{equation}
\begin{cases}
dX(t)= - b(X(t))dt+\sigma(X(t))dW(t) - \rho(t)d\xi(t), \; t\ge 0\\
\;X(0)=x
\end{cases}\nonumber
\end{equation}
and $\tau=\inf\{t\ge 0: X^{\rho,\xi}(t)\notin O\}$. We are assuming that $\sigma,b,c$ are smooth on $\overline{O}$ and that the above SDE has 
as unique solution (in law) for each $x\in\overline{O}$ and control process $(\rho,\xi)$.  In general,  $X$ will not have
continuous sample paths and so it is regarded as a ``singularly" controlled process. Therefore, we say that $u$ is the value function of a
problem of stochastic singular control.

\par W. Fleming and H. Soner have shown that if the value function $u$ satisfies a natural dynamic programming principle, then $u$ is a viscosity solution of a 
Hamilton-Jacobi-Bellman (HJB) equation of the 
form \eqref{mainPDE} (Theorem 5.1, section VIII.5 in \cite{FS}). This result, restated below, provides the connection between equation \eqref{mainPDE} and stochastic singular control. 

\begin{thm}\label{SonerThm}
Assume that for each stopping time $\theta$ (with respect to the filtration generated by $W$) and $x\in \overline{O}, $ 
\begin{eqnarray}
u(x)&=&\inf_{\rho, \xi}\E^x\left\{e^{-\int^{\tau\wedge\theta}_{0}c(X^{\rho,\xi}(s))ds}u(X^{\rho,\xi}(\tau\wedge\theta)) \right. \nonumber \\
 &&\hspace{1in}\left. + \int^{\tau\wedge\theta}_{0}e^{-\int^{t}_{0}c(X^{\rho,\xi}(s))ds}\left[f(X^{\rho,\xi}(t))dt+\ell(\rho(t))d\xi(t)\right]\right\}.\nonumber
\end{eqnarray}
Then the value function $u$ is a viscosity solution of HJB equation
\begin{equation}\label{HJBeq}
\begin{cases}
\max\left\{-\frac{1}{2}\sigma(x)\sigma^t(x)\cdot D^2u + b(x)\cdot Du + c(x)u - f(x), H(Du)\right\}=0, \; x\in O\\
\hspace{4.28in} u=0, \; x\in \partial O
\end{cases}
\end{equation}
where 
\begin{equation}\label{Happ}
H(p)=\max_{|v|=1}\left\{p\cdot v - \ell(v)\right\}, \quad p\in \R^n.
\end{equation}
\end{thm}
In particular, $u$ is a viscosity solution of \eqref{mainPDE} with 
\begin{equation}\label{aApplication}
a(x):=\frac{1}{2}\sigma(x)\sigma^t(x), \; x\in\overline{O}.
\end{equation}
In view of Theorem \ref{mainTheorem}, we have the following corollary which asserts that singular stochastic control problems as described above have $C^{1}$ value functions.
\begin{cor}\label{AppTheorem}
Assume the hypotheses of Theorem \ref{SonerThm}, that $a:\overline{O}\rightarrow {\cal S}(n)$ defined in \eqref{aApplication} satisfies the uniform ellipticity condition \eqref{aElliptic}, and that $0\in \R^n$ is an interior point of the convex set $K$ associated with $\ell$. \\
(i) Then $u$ given by \eqref{valuefun} is the unique viscosity solution of \eqref{HJBeq}. In particular, $u\in C^{1,\alpha}_{\text{loc}}(O)$ for each $0<\alpha<1$.
\\ (ii) Further suppose $\ell$ is the support function of $K:=\{p\in \R^n: G(p)\le 0\}$ where  $G$ satisfies \eqref{Hcond} and \eqref{Hcond2}. Then $u\in C^{1,1}_{\text{loc}}(O)$.
\end{cor}

\begin{proof} $(i)$ If $0$ lies in the interior of $K$,  $B_\delta(0)\subset K$ for some $\delta>0.$ It is immediate from \eqref{ellapp} that $\ell\ge \delta$ and in particular that $H(0)\le-\delta <0$.  The conclusion follows at once from Theorem \ref{mainTheorem}. 

\par $(ii)$ A standard fact about the support function $\ell$ of a closed, convex subset $K\subset \R^n$ is that
$$
K= \bigcap_{|v|=1}\left\{p\in \R^n: v\cdot p\le \ell(v)\right\}.
$$
(Theorem 8.24 in \cite{Rock}).  Therefore, we have by hypothesis
$$
\{p\in \R^n: G(p)\le 0\} = \bigcap_{|v|=1}\left\{p\in \R^n: v\cdot p\le \ell(v)\right\}.
$$
In view of \eqref{Happ}, it is also plain that
$$
\{p\in \R^n: H(p)\le 0\} = \bigcap_{|v|=1}\left\{p\in \R^n: v\cdot p\le \ell(v)\right\}.
$$
Consequently, $u$ solves the PDE
$$
\max\left\{-\frac{1}{2}\sigma(x)\sigma^t(x)\cdot D^2u + b(x)\cdot Du + c(x)u - f(x), G(Du)\right\}=0, \; x\in O.
$$
From the assumptions made on $G$, it follows from Theorem \ref{mainTheorem} that $u\in C^{1,1}_{\text{loc}}(O).$
\end{proof}

\section{Comparison principle}\label{compprin}
In this section, we will verify a fundamental comparison principle among viscosity sub- and supersolutions of the PDE
$$
\max\{Lu-f,H(Du)\}=0.
$$
Although it is now well known how to establish comparison principles of elliptic nonlinear PDE, the problem we have must 
be individually addressed because of the gradient constraint.  In what follows, we use the notation and several basic
results found in standard sources on viscosity solutions such as \cite{C,CIL} and the references therein.  Throughout this paper, 
all PDE and partial differential inequalities will be interpreted in the viscosity sense. Therefore, we may sometimes omit the term
``viscosity" when we mention solutions, subsolutions, and supersolutions.


\begin{prop}\label{comparisonProp}
Assume $u$ is a subsolution of \eqref{mainPDE} and $v$ is a supersolution of \eqref{mainPDE}. If 
$$
u\le v \; \text{on $\partial O$}\quad \text{and} \quad u\in L^\infty(\partial O),
$$
then $u\le v$ in $\overline{O}$. 
\end{prop}

\noindent {\it Formal Proof}. Before proving the above proposition, we give a formal proof (i.e. assuming $u,v\in C^2(O))$ that will help motivate a rigorous argument. 
Fix $\epsilon \in (0,1)$ and set

$$
w^\epsilon(x)=\epsilon u(x) - v(x), \quad x\in \overline{O}.
$$
$w^\epsilon\in USC(\overline{O})$ and thus achieves its maximum at some $x_\epsilon\in \overline{O}$. If $x_\epsilon \in \partial O$, then 
$$
w^\epsilon(x_\epsilon)=-(1-\epsilon)u(x_\epsilon)+u(x_\epsilon)-v(x_\epsilon)\le -(1-\epsilon)u(x_\epsilon)\le (1-\epsilon)|u|_{L^\infty(\partial O)}.
$$
If $x_\epsilon \in O$, then by calculus 
$$
\begin{cases}
0=Dw^\epsilon(x_\epsilon)=\epsilon Du(x_\epsilon) - Dv(x_\epsilon)\\
0\ge D^2w^\epsilon(x_\epsilon)=\epsilon D^2u(x_\epsilon) - D^2v(x_\epsilon)
\end{cases}.
$$
As $H(Du(x_\epsilon))\le 0$,  
\begin{equation}\label{Dhv0}
H(Dv(x_\epsilon))=H(\epsilon Du(x_\epsilon))\le \epsilon H(Du(x_\epsilon))+(1-\epsilon)H(0)<0
\end{equation}
by \eqref{Hcond}.  In particular, since $v$ is a supersolution, we have that 

$$
Lv(x_\epsilon)-f(x_\epsilon)\ge 0.
$$
Therefore,
\begin{eqnarray}
 c(x_\epsilon)w(x_\epsilon)  &\le & L(\epsilon u - v)(x_\epsilon) \nonumber \\
                                    &\le &  - (1-\epsilon)f(x_\epsilon)\nonumber\\
                                    &\le & 0 \nonumber
\end{eqnarray}
and hence $w^\epsilon(x_\epsilon)\le 0$.  In either case, $w^\epsilon\le C(1-\epsilon)$, and letting $\epsilon\rightarrow 1^-$ gives $u\le v.$ 
\begin{flushright}
$\Box$
\end{flushright}

\begin{proof} (of the proposition)
1. Fix $\epsilon \in (0,1)$ and set 
$$
w^{\eta}(x,y)=\epsilon u(x) - v(y) - \frac{1}{2\eta}|x-y|^2, \quad x,y \in \overline{O}
$$
for $\eta>0.$  $w^\eta\in USC(\overline{O}\times \overline{O})$ and so has a maximum at some point $(x_\eta, y_\eta)\in\overline{O}\times \overline{O}$.  As $\overline{O}$ is compact, 
$(x_\eta, y_\eta)$ has a limit point of the form $(x_\epsilon, x_\epsilon)$ through some sequence of $\eta\rightarrow 0^+$, where $x_\epsilon$ is a maximizing point of $x\mapsto \epsilon u(x)-v(x)$, and through this sequence of $\eta\rightarrow 0^+$
\begin{equation}\label{LimLem}
\frac{|x_\eta-y_\eta|^2}{\eta}\rightarrow 0
\end{equation}
(Lemma 3.1 in \cite{CIL}).  

\par 2. If $x_\epsilon\in \partial O$, we have from the definition of $w^\eta$
and our assumptions that
$$
\epsilon u(x) - v(x)\le \epsilon u(x_\epsilon) - v(x_\epsilon)=-(1-\epsilon)u(x_\epsilon) + u(x_\epsilon)-v(x_\epsilon)\le C(1-\epsilon),\quad x\in \overline{O}.
$$
Now we assume that $x_\epsilon\in O$ and without any loss of generality that $(x_\eta,y_\eta)\in O\times O$ for $\eta>0$.
According to the Theorem of Sums (Theorem 3.2 in \cite{CIL}), for each $\rho >0$ there are $X,Y\in {\mathcal S}(n)$ such that 
$$
\left(\frac{x_\eta - y_\eta}{\eta}, X\right)\in \overline{J}^{2,+}(\epsilon u)(x_\eta)
$$
$$
\left(\frac{x_\eta - y_\eta}{\eta}, Y\right)\in \overline{J}^{2,-}v(y_\eta)
$$
and 
\begin{equation}\label{FirstMatIneq}
\left(\begin{array}{cc}
X & 0 \\
0 & -Y
\end{array}\right)\le A + \rho A^2.
\end{equation}
Here 
$$
A=D^2\frac{|x-y|^2}{2\eta} {\Big |}_{x=x_\eta, y=y_\eta}=\frac{1}{\eta}\left(\begin{array}{cc}
I & -I \\
-I & I
\end{array}\right)
$$
and $I$ is the $n\times n$ identity matrix. In particular, choosing $\rho=\eta$ in \eqref{FirstMatIneq} implies the matrix inequality
\begin{equation}\label{IshiiIneq}
\left(\begin{array}{cc}
X & 0 \\
0 & -Y
\end{array}\right)
\le \frac{3}{\eta}
\left(\begin{array}{cc}
I & -I \\
-I & I
\end{array}\right).
\end{equation}

\par 3. Since $u$ is a viscosity subsolution 
\begin{equation}\label{usubsoln}
\max\left\{ - a(x_\eta)\cdot \frac{X}{\epsilon} +b(x_\eta)\cdot \frac{x_\eta -y_\eta}{\epsilon\eta} +c(x_\eta)u(x_\eta)- f(x_\eta), H\left(\frac{x_\eta-y_\eta}{\epsilon\eta}\right) \right\}\le 0, 
\end{equation}
and since $v$ is a viscosity supersolution 
\begin{equation}\label{vsupsoln}
\max\left\{ - a(y_\eta)\cdot Y +b(x_\eta)\cdot \frac{x_\eta -y_\eta}{\eta}+c(y_\eta)v(y_\eta) - f(y_\eta), H\left(\frac{x_\eta-y_\eta}{\eta}\right) \right\}\ge 0. 
\end{equation}
As $H\left(\frac{x_\eta-y_\eta}{\epsilon\eta}\right)\le 0$, we have by \eqref{Hcond} that 
$$
H\left(\frac{x_\eta-y_\eta}{\eta}\right)=H\left(\epsilon\frac{x_\eta-y_\eta}{\epsilon\eta}\right)\le \epsilon H\left(\frac{x_\eta-y_\eta}{\epsilon\eta}\right) + (1-\epsilon)H(0)<0.
$$
By \eqref{vsupsoln}, 
\begin{equation}\label{vsupsoln2}
a(y_\eta)\cdot Y +b(x_\eta)\cdot \frac{x_\eta -y_\eta}{\eta}+c(y_\eta)v(y_\eta) - f(y_\eta)\ge 0.
\end{equation}
Combining  \eqref{usubsoln} and \eqref{vsupsoln2} gives, 
\begin{eqnarray}\label{HopeFinIneq}
\epsilon c(x_\eta)u(x_\eta)  - c(y_\eta)v(y_\eta) & \le & a(x_\eta)\cdot X - a(y_\eta)\cdot Y + (b(x_\eta)-b(y_\eta))\cdot \frac{x_\eta - y_\eta}{\eta} \nonumber \\
                                                                                                  && + \epsilon f(x_\eta)-f(y_\eta) \nonumber \\
                                                                                                  &\le &  a(x_\eta)\cdot X - a(y_\eta)\cdot Y+ \text{Lip}(b)\frac{|x_\eta - y_\eta|^2}{\eta}+ \text{Lip}(f)|x_\eta-y_\eta|. \nonumber \\
\end{eqnarray}
Note $x\mapsto a^{1/2}(x)$ (the {\it unique} positive square root of $a(x)$) is Lipschitz continuous since $x\mapsto a(x)$ is Lipschitz and $a\ge \gamma>0$; indeed 
$$
\text{Lip}(a^{1/2})\le \frac{\text{Lip}(a)}{2\gamma}. 
$$ Also note that the 
$2n\times 2n $ matrix
\begin{equation}
\left(\begin{array}{cc}
a^{1/2}(x_\eta)a^{1/2}(x_\eta) & a^{1/2}(x_\eta)a^{1/2}(y_\eta)\\
a^{1/2}(y_\eta)a^{1/2}(x_\eta) & a^{1/2}(y_\eta)a^{1/2}(y_\eta)
\end{array}\right) \nonumber
\end{equation}
is non-negative definite, and by \eqref{IshiiIneq}

\begin{eqnarray}
a(x_\eta)\cdot X -a(y_\eta)\cdot Y &=& \tr\left[a(x_\eta)X - a(y_\eta)Y\right] \nonumber \\
&=&\tr\left[ \left(\begin{array}{cc}
a^{1/2}(x_\eta)a^{1/2}(x_\eta) & a^{1/2}(x_\eta)a^{1/2}(y_\eta)\\
a^{1/2}(y_\eta)a^{1/2}(x_\eta) & a^{1/2}(y_\eta)a^{1/2}(y_\eta)
\end{array}\right)\left(\begin{array}{cc}
X & 0 \\
0 & -Y
\end{array}\right)\right] \nonumber \\
&\le & 
\tr\left[ \left(\begin{array}{cc}
a^{1/2}(x_\eta)a^{1/2}(x_\eta) & a^{1/2}(x_\eta)a^{1/2}(y_\eta)\\
a^{1/2}(y_\eta)a^{1/2}(x_\eta) & a^{1/2}(y_\eta)a^{1/2}(y_\eta)
\end{array}\right)
\frac{3}{\eta}
\left(\begin{array}{cc}
I & -I \\
-I & I
\end{array}\right)\right]\nonumber \\
&\le & \frac{3}{\eta}\tr\left[(a^{1/2}(x_\eta)-a^{1/2}(y_\eta))((a^{1/2}(x_\eta)-a^{1/2}(y_\eta))\right] \nonumber \\
&\le & \frac{3\text{Lip}(a)^2}{2\gamma^2}\frac{|x_\eta-y_\eta|^2}{2\eta}. \nonumber
\end{eqnarray}  
By \eqref{HopeFinIneq},
\begin{equation}\label{finalManipEta}
\epsilon c(x_\eta)u(x_\eta)  - c(y_\eta)v(y_\eta) \le \left(\frac{3\text{Lip}(a)^2}{2\gamma^2}+2\text{Lip}(b)\right)\frac{|x_\eta-y_\eta|^2}{2\eta} + \text{Lip}(f)|x_\eta-y_\eta|.
\end{equation}
\par 4. Let $(x_\epsilon, x_\epsilon)$ be a limit point of $(x_\eta, y_\eta)$ through as sequence of $\eta\rightarrow 0^+.$ If $x_\epsilon\in \partial O$, we have from our remarks above that
$$
\epsilon u(x_\epsilon)-v(x_\epsilon)\le C(1-\epsilon).
$$
If $x_\epsilon\in O$, we let $\eta\rightarrow 0^+$ through the appropriate subsequence in \eqref{finalManipEta} and recall \eqref{LimLem} to arrive at 
$$
c(x_\epsilon)(\epsilon u(x_\epsilon)-v(x_\epsilon))\le 0.
$$
This inequality implies $\epsilon u(x_\epsilon)-v(x_\epsilon)\le 0$, and so in either case,
$$
\epsilon u(x)-v(x)\le \epsilon u(x_\epsilon)-v(x_\epsilon)\le C(1-\epsilon), \quad x\in \overline{O}.
$$
We conclude by letting $\epsilon\rightarrow 1^-.$
\end{proof}

\begin{rem}\label{PerronRemark}
With a comparison principle in hand, we can now employ a routine application of Perron's method to obtain the existence of solutions.   Indeed, observe that 
\begin{equation}\label{subsoln} 
\underline{u}\equiv 0 \nonumber
\end{equation}
is a subsolution of \eqref{mainPDE}; and $\bar{u}$, the unique solution of 
\begin{equation}\label{Eq4SupSoln}
\begin{cases}
Lv=f, \quad x\in O\\
\hspace{.125in} v=0, \quad x\in \partial O
\end{cases},
\end{equation}
is a supersolution of \eqref{mainPDE}. Therefore, Perron's method applies from which we conclude
$$
u(x):=\sup\left\{ w(x): \underline{u}\le w \le \overline{u}, \; \text{$w$ is a subsolution of \eqref{mainPDE}}\right\}
$$
is a viscosity solution (Theorem 4.1 in \cite{CIL}).  By the comparison principle, this solution is unique.  For the remainder of this paper, we pursue the 
regularity of solutions of \eqref{mainPDE}.
\end{rem}

\section{Penalization method}\label{regularity}
In this section, we analyze solutions of the penalized equation \eqref{penalizedPDE}

\begin{equation}
\begin{cases}
Lu^\epsilon + \beta_\epsilon(H(Du^\epsilon))=f, & x\in O\\
\hspace{1.15in}u^\epsilon = 0, & x\in \partial O
\end{cases},\nonumber
\end{equation}
where $(\beta_\epsilon)_{\epsilon>0}$ is a family of functions $(\beta_\epsilon)_{\epsilon>0}$ satisfying
\begin{equation}\label{betaProp}
\begin{cases}
\beta_\epsilon \in C^\infty(\R) \\
\beta_\epsilon = 0, \; z\le 0 \\
\beta_\epsilon >0, \; z>0\\ 
\beta_\epsilon'\ge 0\\
\beta_\epsilon''\ge 0 \\
\beta_\epsilon(z)=\frac{z-\epsilon}{\epsilon}, \quad z\ge 2\epsilon
\end{cases}.
\end{equation}
For each $\epsilon>0$, we think of $\beta_\epsilon$ as a type of smoothing of $z\mapsto (z/\epsilon)^+$; for small $\epsilon$, we think of $\beta_\epsilon$ as a smooth approximation
of the set valued mapping
$$
\beta_0(t)
=
\begin{cases}
\{0\}, \quad &t< 0\\
[0,\infty], \quad &t=0
\end{cases}.
$$
It will be a standing assumption that such a family of functions satisfying \eqref{betaProp} exists. The reason for using this approximation is from the following intuition. 
Since the values of $\beta_\epsilon(H(Du^\epsilon))$ can be large when $H(Du^\epsilon)>0$ and $\epsilon$ small, solutions will seek to satisfy  $H(Du^\epsilon)\le0$ and,
in this sense, become closer to satisfying equation \eqref{mainPDE}.    

\par In our analysis of \eqref{penalizedPDE}, we make the following special assumptions on $H$
\begin{equation}\label{NewHcond}
\begin{cases}
H\in C^2(\R^n)\\
D^2H \ge \theta, \; \text{for some $\theta\in (0,1)$}\\
\sup_{p\in \R^n}|D^2H(p)|<\infty.
\end{cases}.
\end{equation}
Admittedly, Theorem \ref{mainTheorem} addresses a much larger class of gradient constraints $H$. However, we will see in the next section that the assumptions made above can be relaxed by replacing a general $H$ with (a smoothing of) an appropriate inf-convolution.

\par Notice that \eqref{penalizedPDE} is a semi-linear, uniformly elliptic PDE with smooth coefficients.  By our growth assumptions on $\beta_\epsilon$ and \eqref{NewHcond} $\beta_\epsilon(H(p))$ grows at most quadratically as $|p|\rightarrow \infty$, for each $\epsilon>0$. It follows that \eqref{penalizedPDE} has a unique, classical solution $u^\epsilon$ (Theorem 15.10 in \cite{GT}).   Our goal is to derive $W^{2,p}$ estimates $(1\le p\le \infty)$ on solutions that are independent of all $\epsilon>0$ and small. Such estimates would aid us in proving that a subsequence of $u^\epsilon$ converges to $u$, the solution of \eqref{mainPDE}, in  $C^{1}_{\text{loc}}(O)$ as $\epsilon\rightarrow 0^+$.  To this end, our main result concerning the penalization method is as follows. 

\begin{prop}\label{PenalProp} Assume that $O'\subset \subset O$, $1< p<\infty$ and $0<\epsilon<\theta$.  We have the following estimates:  \\
$(i)$ 
$$
|u^\epsilon|_{W^{1,\infty}(O)}\le C
$$
for a universal constant $C$;\\
(ii)
$$
|u^\epsilon|_{W^{2,p}(O')}\le C_1,
$$
for some $C_1$ depending on $p$ and $O'$;\\
$(iii)$ 
$$
|u^\epsilon|_{W^{2,\infty}(O')}\le C_2,
$$
for some $C_2$ depending on $\theta$, $O'$,  and
$$
\max_{|p|\le C}|D^2H(p)|.
$$
\end{prop}
The proof Proposition \ref{PenalProp} is accomplished through the following sequence of lemmas, and we will obtain the desired estimates by employing the Bernstein method.

\begin{lem}
There is a constant $C$ such that 
$$
|u^\epsilon(x)|\le C, \quad x\in \overline{O}
$$
for $\epsilon >0.$
\end{lem}

\begin{proof}
Let $\bar{u}$ be the unique smooth solution of \eqref{Eq4SupSoln}.  As $\bar{u}$ is a supersolution of equation \eqref{penalizedPDE}, $u^\epsilon\le \bar{u}$; while $u^\epsilon\ge 0$, since 
$\underline{u}: x\mapsto 0$ is a subsolution of \eqref{penalizedPDE}.  Hence, $0\le u^\epsilon\le \bar{u}.$
\end{proof}
An immediate corollary of the above proof is

\begin{cor}
There is a constant $C$ such that 
$$
|Du^\epsilon(x)|\le C, \quad x\in\partial O
$$
for $\epsilon >0.$
\end{cor}

\begin{proof}
By the proof of the previous lemma, we have $0\le u^\epsilon\le \bar{u}$ with equality holding 
on $\partial O.$ Thus,
$$
\frac{\partial \bar{u}(x)}{\partial \nu}\le \frac{\partial u^\epsilon(x)}{\partial \nu}\le 0, \quad x\in\partial O
$$
where $\nu$ is the outward normal on $\partial O$ (which is assumed to be smooth).
\end{proof}

\begin{lem}\label{uepsGrad}
There is a constant $C$  such that 
$$
|Du^\epsilon(x)|\le C, \quad x\in \overline{O}
$$
for $0< \epsilon <\theta.$
\end{lem}

\begin{proof}
1. It suffices to bound the function 
$$
v^\epsilon(x)=|Du^\epsilon(x)|^2-\lambda u^\epsilon(x), \quad x\in \overline{O}
$$
from above, for some universal (that is, $\epsilon$ independent) constant $\lambda >0$.  To this end, we suppress $\epsilon$ dependence, function arguments and make use of equation \eqref{penalizedPDE} to compute the following identity
\begin{align}\label{vNotComp}
a\cdot D^2v &= 2 a D^2u\cdot D^2u - 2\sum^n_{i,j=1}u_{x_ix_j}Du\cdot Da_{ij} + Dv\cdot (\beta' DH + b) \nonumber \\
& +2 Du\cdot ( uDc + cDu+ Db\cdot Du -Df) -\lambda(cu + b\cdot Du -f) \nonumber \\
& + \lambda (\beta' Du\cdot DH - \beta).
\end{align}
\par 2. Since  $H$ is uniformly convex with $D^2H\ge \theta$ and $H(0)\le 0$, 
$$
p\cdot DH(p) - H(p)\ge \theta |p|^2/2, \quad p\in \R^n.
$$
Also note that, $\beta_\epsilon(z)\le z\beta_\epsilon'(z)$ for all $z\in \R$ which implies 
$$
\beta' Du\cdot DH - \beta\ge \beta' (Du\cdot DH - H)\ge \beta'\frac{\theta}{2}|Du|^2
$$
Using this inequality, the previous lemma and the uniform ellipticity of $a$ \eqref{aElliptic} we have from \eqref{vNotComp} the estimate
\begin{equation}\label{vComp}
a\cdot D^2v \ge  - C |Du|^2 -C + Dv\cdot (\beta' DH + b) + \frac{\lambda\theta \beta'}{2}  |Du|^2.
\end{equation}

\par 3. Let $x_0\in \overline{O}$ be a maximizing point for $v$. If $x_0\in \partial O$, a bound on $|Du(x_0)|^2$ that is independent of $\epsilon>0$
is immediate from the previous corollary. If $x_0 \in O$, then 
$$
Dv(x_0)=0, \quad a(x_0)\cdot D^2v(x_0)\le 0.
$$
Now, if $\beta'=\beta'(H(Du(x_0)))\le 1/\theta< 1/\epsilon$, then $\beta=\beta(H(Du(x_0)))\le 1$ by \eqref{betaProp}.  In particular, $H(Du(x_0))\le 2\epsilon\le 2\theta\le 2$ which implies a bound on $|Du(x_0)|$ independent of $\epsilon \in (0,\theta)$. If $\beta'(H(Du(x_0)))\ge 1/\theta$, \eqref{vComp} gives 
\begin{equation}
0 \ge  - C |Du(x_0)|^2 -C +  \frac{\lambda}{2} |Du(x_0)|^2, \nonumber
\end{equation}
which implies a bound on $|Du(x_0)|^2$ independent of $\epsilon\in (0,\theta)$, for $\lambda>0$ chosen large enough. 
\end{proof}

\begin{lem}\label{betaLem}
Let $O'\subset\subset O$ . There is a constant $C_1$ depending on $O'$ such that 
$$
0\le \beta_\epsilon(H(Du^\epsilon(x)))\le C_1, \quad x\in O'
$$
for $0<\epsilon<\theta$.
\end{lem}
\begin{rem}
To simplify the arguments given below, we assume 
$$
b\equiv 0,\quad\text{and}\quad c\equiv \delta >0.
$$
It is straightforward to verify that incorporating more general coefficients $b$ and $c$ is merely technical and no new issues arise. 
\end{rem}

\begin{proof}
1.  It suffices to bound 
$$
v^\epsilon(x) :=\eta(x)\beta_\epsilon(H(Du^\epsilon(x)))+\frac{1}{2}|Du^\epsilon(x)|^2, \quad x\in \overline{O}
$$
for each $\eta\in C^\infty_c(O)$, $0\le \eta\le 1.$  As before, we will omit the $\epsilon$ dependence 
of $u^\epsilon$ and $v^\epsilon$, omit the arguments of functions, write $\beta$ for $\beta_\epsilon(H(Du^\epsilon))$ and  compute the following identity
for $a\cdot D^2v$

\begin{align}\label{betaaD2u}
a\cdot D^2v & = (a\cdot D^2\eta)\eta + \beta'DH\cdot Dv + \beta '' a D^2uDH\cdot D^2uDH  \nonumber \\
& \beta'\left\{  \eta\left(a\cdot D^2u D^2H D^2u + DH\cdot D(\delta u -f) - \sum^n_{i,j=1}u_{x_ix_j}Da_{ij}\cdot DH \right) \right. \nonumber \\
& \quad \quad  -\beta DH\cdot D\eta - DH\cdot D^2u Du + D^2u Du\cdot Du \text{\Huge{\}}}  \nonumber \\
& +aD^2u\cdot D^2u + Du\cdot D(\delta u -f) - \sum^n_{i,j=1}u_{x_ix_j}Da_{ij}\cdot Du.
\end{align}
Below, we will make use of the following inequalities  
\begin{equation}\label{betaaD2uINEQ}
\begin{cases}
\beta '' a D^2uDH\cdot D^2uDH\ge 0\\
a\cdot D^2u D^2H D^2u\ge \gamma\theta |D^2u|^2\\
aD^2u\cdot D^2u\ge\gamma|D^2u|^2\\
0\le \beta\le C\{1+|D^2u|\}
\end{cases},
\end{equation}
which follow from PDE \eqref{penalizedPDE} and our various assumptions on $\beta$, $a$ and H \eqref{betaProp}, \eqref{aElliptic}, \eqref{NewHcond}. 

\par 2.  Let $x_0\in \overline{O}$ be a maximizing point for $v$.  We may as well assume that $x_0\in O$; otherwise,  $v\le |Du(x_0|^2/2\le C$. In this case, we have 
$$
Dv(x_0)=0 \quad \text{and}\quad 0\ge a(x_0)\cdot D^2v(x_0).
$$
From these inequalities, \eqref{betaaD2u} and \eqref{betaaD2uINEQ}, it is straightforward to derive the following inequality
\begin{equation}\label{Realfinalad2v}
0 \ge \frac{\gamma}{2}|D^2u|^2- C-C\beta' \left\{|D^2u|+1\right\} 
\end{equation}
where $C$ denotes various constants that are independent of $\epsilon\in (0,\theta)$  (but may depend on  $\eta(x_0), D\eta(x_0), D^2\eta(x_0)$).  All functions in \eqref{Realfinalad2v} and in the rest of this proof are evaluated at $x_0.$

\par If 
$$
\beta '\le 1<1/\epsilon\quad\text{or}\quad \frac{\gamma}{2}|D^2u|^2- C\le 0,
$$
then we have a uniform upper bound on $\beta=\beta(H(Du(x_0)))$ and also the desired upper bound on $v(x_0)$. Otherwise,   \eqref{Realfinalad2v} implies
$$
0\ge\frac{1}{2}\gamma |D^2u|^2- C - C(|D^2u| +1)
$$ 
and in particular $|D^2u(x_0)|\le C$. By the last inequality in \eqref{betaaD2uINEQ}, $v\le C_1$ for some $C_1$ independent of $\epsilon\in (0,\theta)$ and only depending on $O'$.
\end{proof}
\noindent 
Thus far, we have established part $(i)$ of Propostion \ref{PenalProp}; the conclusion follows at once from the previous lemma.  The previous lemma will also aid us in obtaining a pointwise bound on the second derivatives of $u^\epsilon$ and thereby establish part $(ii)$ of Propostion \ref{PenalProp}. To this end, we adapt the approach by M. Wiegner \cite{W}.  

\begin{lem}
Let $O'\subset\subset O$ and $C>0$ as in Lemma \ref{uepsGrad}.  There is a constant $C_2$  depending on $\theta$, $O'$, $\max_{|p|\le C}|D^2H(p)|$
such that
$$
|u^\epsilon|_{W^{2,\infty}(O')}\le C_2
$$
for each $0<\epsilon <\theta$.
\end{lem}

\begin{proof}
1. It is sufficient to bound, for each $\eta \in C^\infty_c(O)$ with $0\le \eta \le 1$, the quantity
$$
M_\epsilon:=\max_{x\in \overline{O}}\{\eta(x)|D^2u^\epsilon(x)|\}
$$
for all $0<\epsilon<\theta$. With this in mind, we shall bound the related function
$$
v^\epsilon(x)=\frac{1}{2}\eta(x)^2|D^2u^\epsilon(x)|^2 + \lambda\eta(x)\beta(H(Du^\epsilon(x)))+\frac{\mu}{2}|Du^\epsilon(x)|^2, \quad x\in \overline{O}
$$
from above. Here $\lambda, \mu$ are constants that will be chosen below.  As in previous proofs, we shall omit the $\epsilon$
dependence of $u^\epsilon,v^\epsilon$ and their derivatives and many times we will write $\beta$ for $\beta_\epsilon(H(Du^\epsilon))$.  

\par Direct computation gives
$$\begin{cases}
v_{x_i}=\eta\eta_{x_i}|D^2u|^2+ \eta^2D^2u\cdot D^2u_{x_i}+ \lambda(\eta_{x_i}\beta  +\eta \beta' DH\cdot Du_{x_{i}}) + \mu Du\cdot Du_{x_i}\\
v_{x_ix_j}=(\eta\eta_{x_ix_j}+\eta_{x_i}\eta_{x_j})|D^2u|^2+ 2\eta\eta_{x_i}D^2u\cdot D^2u_{x_j} +2\eta\eta_{x_j}D^2u\cdot D^2u_{x_i} + \\ 
\hspace{.5in}\eta^2 (D^2u_{x_i}\cdot D^2u_{x_j}+D^2u\cdot D^2u_{x_ix_j})+ \lambda\left[\eta_{x_ix_j}\beta +\eta_{x_i}\beta' DH\cdot Du_{x_j}+\eta_{x_j}\beta' DH\cdot Du_{x_i}  \right.\\
\hspace{.5in} \left.\eta\left\{\left(\beta'' (DH\cdot Du_{x_i})(DH\cdot Du_{x_j}) + \beta' (D^2HDu_{x_i}\cdot Du_{x_j} + DH\cdot Du_{x_ix_j})\right)
\right\}\right]\\
\hspace{.5in}  + \mu\left\{Du_{x_i}\cdot Du_{x_j} + Du\cdot Du_{x_ix_j} \right\}
\end{cases}.
$$
Using the above expressions for $v_{x_i}$ and $v_{x_ix_j}$ and the fact that $u$ solves the PDE \eqref{penalizedPDE}, it is straightforward to verify the following identity
\begin{eqnarray}\label{hessD2v}
a\cdot D^2v&= & (\eta a\cdot D^2\eta + aD\eta\cdot \eta)|D^2u|^2+ 4\sum^n_{i,j=1}a_{ij}\eta_{x_i}D^2u\cdot \eta D^2u_{x_j} +  \nonumber \\
&&  \eta^2\left(aD^3u\cdot D^3u+D^2u\cdot D^2(\delta u-f) -\sum^{n}_{k,l=1}u_{x_k x_l}\left[2 (a_{x_k}\cdot D^2u_{x_\ell}) + (a_{x_k x_l}\cdot D^2u)\right]\right) \nonumber \\
&& + \mu \left\{ aD^2u\cdot D^2u +Du\cdot D(\delta u -f) -\sum^n_{i,j=1}u_{x_ix_j}Da_{ij}\cdot Du\right\}+\lambda (a\cdot D^2\eta)\beta  \nonumber\\
&& +\beta'' \left(\eta^2 D^2u(D^2uDH)\cdot(D^2uDH) + \eta\beta''\lambda \gamma|D^2uDH|^2\right) + \beta'DH\cdot Dv + \nonumber\\
&& \beta'\left[\lambda\left(aD\eta\cdot D^2uDH + \eta\left\{a\cdot D^2uD^2HD^2u + DH\cdot D(\delta u -f)   -\sum^n_{i,j=1}u_{x_ix_j}Da_{ij}\cdot DH\right\}\right.\right.    \nonumber \\
&& \left. -\beta D\eta \cdot DH\huge{\text{)}} -\eta|D^2u|^2DH\cdot D\eta +\eta^2D^2u\cdot D^2u D^2H D^2u\right].
\end{eqnarray}

\par 2.  Let $x_0$ be a maximizing point for $v$. If  $x_0\in \partial O$, then $v\le v(x_0)\le \mu |Du(x_0)|^2\le C$. This of course implies $M_\epsilon^2\le v(x_0)\le C$ as desired. Now suppose that $x_0\in O$, so that 

$$
Dv(x_0)=0\quad \text{and}\quad a(x_0)\cdot D^2v(x_0)\le 0.
$$
Evaluating identity \eqref{hessD2v} at $x_0$, employing the assumed matrix inequalities $a\ge \gamma$ \eqref{aElliptic} and $D^2H\ge \theta $ \eqref{NewHcond}, and repeatedly using the Cauchy-Schwarz inequality gives us 
\begin{eqnarray}\label{hessD2v2}
0 & \ge &  - C -C |D^2u|^2 + \mu \left(\frac{1}{2} \gamma |D^2u|^2 -C\right)  \nonumber \\
 && + \eta\beta''|D^2uDH|^2(\gamma\lambda - \eta|D^2u|) + \nonumber\\
 &&+\beta'\left\{ \lambda\left(\frac{1}{2}\eta \gamma\theta|D^2u|^2  - C-C|D^2u|\right) -  C\eta|D^2u|^2 -C'\eta^2|D^2u|^3\right\}. \nonumber
\end{eqnarray}
Here $C$  is denotes various constants that are independent of $\epsilon$, but may depend on $|DH(Du(x_0))|$, $\eta(x_0)$, $D\eta(x_0)$, $D^2\eta(x_0)$ and 
$$
C':=\max_{|p|\le C}|D^2H(Du(p))|\ge |D^2H(Du(x_0))|.
$$
Observe that $C'$ is a universal constant by assumption \eqref{NewHcond}.

 \par 3. Set 
$$
\lambda=\lambda_\epsilon:= \alpha M_\epsilon
$$
where $\alpha> \max\{1/\gamma,C'\}$ is a positive constant. Notice that
\begin{eqnarray}\label{hessD2v22}
0 & \ge &  - C -C |D^2u|^2 + \mu \left(\frac{1}{2} \gamma |D^2u|^2 -C\right)  \nonumber \\
 && + \eta^2|D^2u|\beta''|D^2uDH|^2(\gamma \alpha - 1) + \nonumber\\
 &&+\beta'\left\{ \alpha\eta|D^2u|\left(\frac{1}{2}\eta \gamma\theta|D^2u|^2  - C-C|D^2u|\right) -  C\eta|D^2u|^2 -C'\eta^2|D^2u|^3\right\}\nonumber\\
 & \ge &  - C -C |D^2u|^2 + \mu \left(\frac{1}{2} \gamma |D^2u|^2 -C\right)  \nonumber \\
 &&+\beta'\left\{ \alpha\eta|D^2u|\left(\frac{1}{2}\eta \gamma\theta|D^2u|^2  - C-C|D^2u|\right) -  C\eta|D^2u|^2 -C'\eta^2|D^2u|^3\right\}.
\end{eqnarray}
\par  If for this choice of $\alpha $ 
$$
\alpha\eta|D^2u|\left(\frac{1}{2}\eta \gamma\theta|D^2u|^2  - C-C|D^2u|\right) -  C\eta|D^2u|^2 -C'\eta^2|D^2u|^3\le 0,
$$
then there is a bound on $\eta(x_0)|D^2u(x_0)|$ depending only on $\theta$, $ |DH(Du(x_0))|$, and  $C'$. If the above inequality does not hold, then  \eqref{hessD2v22} implies
$$
0  \ge   - C -C |D^2u|^2 + \mu \left(\frac{1}{2} \gamma |D^2u|^2 -C\right).  
$$
For $\mu$ chosen large enough, this inequality implies a universal bound on $\eta(x_0)|D^2u(x_0)|.$    

\par In all cases, we have bounded $\eta(x_0)|D^2u(x_0)|^2$ from above independently of $0< \epsilon<\theta$ (but perhaps dependent on $\theta$, and  $C'$) and therefore,
$$
M_\epsilon^2\le \max_{\overline{O}}v\le \eta(x_0)|D^2u(x_0)|^2 + CM_\epsilon +C\le C(M_\epsilon + 1).
$$
Consequently,  $M_\epsilon\le C_2$ for some $C_2$ (possibly) depending only on $O'$, $C'$ and $\theta$. 
\end{proof}

\section{Convergence}\label{ConvSec}
In this final section, we complete the proof Theorem \ref{mainTheorem}.  We have already established the existence of a unique viscosity solution in Proposition \ref{comparisonProp}; see Remark \ref{PerronRemark}.  The only issue remaining is the regularity of solutions, which will be handled by our results in the previous section. 
Note however, that we established Proposition \ref{PenalProp} under the extra hypotheses \eqref{NewHcond}, while Theorem \ref{mainTheorem} does make this assumption. The main idea to overcome this technicality is to employ the {\it inf-convolution}   
\begin{equation}\label{InfSupH}
H^t(p):=\inf_{q\in \R^n}\left\{H(q)+\frac{1}{2t}|p-q|^2\right\}, \quad (p,t)\in \R^n\times (0,\infty).
\end{equation}
We make use following lemma and omit the proof as it is elementary. 

\begin{lem}\label{propHt} Assume that $H:\R^n\rightarrow \R$ satisfies \eqref{Hcond}.\\
(i) $H^t(0)<0$, for $t>0.$\\
(ii) $H^t\rightarrow H$ locally uniformly on $\R^n$, as $t\rightarrow 0^+.$\\
(iii) $H^t$ is convex and for each $p,z\in \R^n$ and $t>0$
$$
H^t(p+z)-2H^t(p)+H^t(p-z)\le \frac{|z|^2}{t}.
$$
In particular, $0\le D^2H^t \le 1/t, \; a.e. \; p\in \R^n$.\\\\
Assume that $H$ satisfies \eqref{Hcond2} and $D^2H(p)\ge \theta,$ for a.e. $p\in \R^n$. \\
(iv) For each $t,R>0$, 
$$
H^t(p+z)-2H^t(p)+H^t(p-z)\le \left(|D^2H|_{L^\infty(B_{Q(R,t)+|z|})}\right)|z|^2
$$
for $|p|\le R$ and $z\in \R^n$, where 
\begin{equation}\label{Qfun}
Q(R,t):=2\left(R + t\max_{|w|\le R}|DH(w)|\right).\nonumber
\end{equation}
\\
(v) For each $p,z,\in \R^n, \; t>0$
$$
H^t(p+z)-2H^t(p)+H^t(p-z)\ge\frac{\theta|z|^2}{1+t\theta}.
$$
Hence, $D^2H^t \ge\theta/(1+t\theta), \; a.e.\; p\in \R^n$, $t>0$.
\end{lem}

\begin{proof} (of Theorem \ref{mainTheorem}) 1. First assume that $H$ satisfies \eqref{NewHcond}, so that we can apply Proposition \ref{PenalProp}.   In view of 
the conclusion of this proposition, a routine application of the Arzel\`{a}-Ascoli Theorem implies there is a sequence of $\epsilon_k\rightarrow 0$ such that
$$
\begin{cases}
u^{\epsilon_k}\rightarrow v \quad \text{uniformly in $\overline{O}$}\\
u^{\epsilon_k}\rightarrow v \quad \text{in $C^1_{\text{loc}}(O)$}
\end{cases},
$$
as $k\rightarrow \infty$.  That is, $u^{\epsilon_k}\rightarrow v$ uniformly in $\overline{O}$ and $u^{\epsilon_k}\rightarrow v$ in $C^1(O')$ for each $O'\subset\subset O$, as 
$k\rightarrow \infty.$  It is clear from the above convergence that $v\in C^{1,1}_{\text{loc}}(O)\cap C^{0,1}(O).$ We now {\bf claim}: $v$ is a viscosity solution of \eqref{mainPDE} and therefore has to coincide with $u$ by the uniqueness of viscosity solutions of \eqref{mainPDE}.

\par 2.    Suppose that
$v-\varphi $ has a local maximum at $x_0\in O$ and that $\varphi\in C^2(O)$. We must show 
\begin{equation}\label{vviscsub}
\max\left\{\delta v(x_0)- a(x_0)\cdot D^2\varphi(x_0) - f(x_0), H(D\varphi(x_0)) \right\}\le 0.
\end{equation}
By adding $x\mapsto \frac{\rho}{2} |x-x_0|^2$ to $\varphi$ and later sending $\rho\rightarrow 0$, we may assume that $v-\varphi $ has a {\it strict} local maximum.  Since $u^{\epsilon_k}$ 
converges to $v$ uniformly (for some sequence $\epsilon_k\rightarrow 0$) as $k\rightarrow \infty$, there is a sequence of $x_k$ such that

$$
\begin{cases}
x_k\rightarrow x_0,\quad \text{as}\; k\rightarrow \infty\\
u^{\epsilon_k}-\varphi \;\;\text{has a local maximum at $x_k$}
\end{cases}.
$$
As $u^{\epsilon_k}$ is a smooth solution of \eqref{penalizedPDE}, we have 
$$
\delta u^{\epsilon_k}(x_k)- a(x_k)\cdot D^2\varphi(x_k) + \beta_{\epsilon_k}(H(D\varphi(x_k)))\le f(x_k).
$$
Since $\beta_\epsilon \ge 0$, we can send $k\rightarrow \infty$ to arrive at
$$
\delta v(x_0)- a(x_0)\cdot D^2\varphi(x_0)\le f(x_0). 
$$
By Lemma \ref{betaLem},
$$
0\le  \beta_{\epsilon_k}(H(D\varphi(x_k)))=\beta_{\epsilon_k}(H(Du^{\epsilon_k}(x_k)))\le C,
$$
which necessarily implies that when $k\rightarrow \infty$
$$
H(D\varphi(x_0))\le 0.
$$

\par 3. Now suppose that 
$v-\psi $ has a local minimum at $x_0\in O$ and that $\psi\in C^2(O)$. We must show 
\begin{equation}\label{vviscsup}
\max\left\{\delta v(x_0)- a(x_0)\cdot D^2\psi(x_0) - f(x_0), H(D\psi(x_0)) \right\}\ge 0.
\end{equation}
Arguing as above, we discover there is a sequence $\epsilon_k\rightarrow 0$ as $k\rightarrow \infty$, and $x_k$ such that
$$
\begin{cases}
x_k\rightarrow x_0,\quad \text{as}\; k\rightarrow \infty\\
u^{\epsilon_k}-\psi \;\;\text{has a local minimum at $x_k$}
\end{cases}.
$$
If 
$$
H(D\psi(x_0))\ge 0,
$$
then \eqref{vviscsup} holds. Suppose now that 
$$
H(D\psi(x_0))< 0.
$$
Since $u^\epsilon$ is a smooth solution of \eqref{penalizedPDE}, we have 
\begin{equation}\label{uepsKineq}
\delta u^{\epsilon_k}(x_k)- a(x_k)\cdot D^2\psi(x_k) + \beta_{\epsilon_k}(H(D\psi(x_k)))- f(x_k)\ge 0.
\end{equation}
By the convergence established in part 1 of this proof, $H(D\psi(x_k))= H(Du^{\epsilon_k}(x_k))<0$ for all large enough $k$. Hence,
$$
\lim_{k\rightarrow \infty}\beta_{\epsilon_k}(H(D\psi(x_k)))=0.
$$ 
In this case, the above limit and \eqref{uepsKineq} imply
$$
\max\left\{\delta v(x_0)- a(x_0)\cdot D^2\psi(x_0) - f(x_0), H(D\psi(x_0)) \right\}\ge \delta v(x_0)- a(x_0)\cdot D^2\psi(x_0) - f(x_0)\ge 0.
$$
This argument verifies the claim that $v=u$ and in particular, it proves Theorem \ref{mainTheorem} under 
the hypothesis that $H$ satisfies \eqref{NewHcond}.

\par 4.  Now assume that $H$ satisfies \eqref{Hcond} and define for $t,\rho,\theta \in (0,1)$ the function 
$$
H^{t,\rho, \theta}(p):=\theta |p|^2 + H^{t,\rho}(p), \quad p\in \R^n,
$$
where
$$
H^{t,\rho}(p):=\int_{\R^n}\eta^\rho(p-y)H^t(y)dy, \quad p\in \R^n.
$$
Here $H^t$ is the inf-convolution of $H$ given in \eqref{InfSupH}, $\eta^\rho(x)=\rho^{-n}\eta(x/\rho)$, $\eta\in C^\I_c(\R^n)$, $\int_{\R^n}\eta(x)dx=1$, and $\eta(x)=0$ for $|x|\ge 1.$ Therefore, $H^{t,\rho}\in C^\I(\R^n),$ and as $\rho$ tends to $0$, $H^{\rho,t}$ converges to $H^t$
locally uniformly (consult Appendix C in \cite{Evans} for more on mollification).  

\par By parts $(i)$ and $(iii)$ of Lemma \ref{propHt}, $H^{t,\rho,\theta}$ satisfies \eqref{NewHcond} and (for $\rho$ small enough) \eqref{Hcond}.  Hence, 
$$
\begin{cases}
\max\left\{Lu-f, H^{t,\rho,\theta}(Du)\right\}=0, \; x\in O\\
\hspace{1.74in} u=0, \; x\in \partial O
\end{cases}
$$
has a unique solution $u^{t,\rho,\theta}\in C^{1,1}_{\text{loc}}$.  Moreover, an easy conclusion of the convergence assertion in parts $1-3$ of the current proof gives the estimate
$$
|u^{t,\rho,\theta}|_{C^{1,\alpha}(O')}\le C(O',\alpha).
$$
for each $t,\rho,\theta,\alpha \in (0,1)$ and $O'\subset\subset O$.  It follows that there is are sequences $t_k,\rho_k,\theta_k\rightarrow 0$ and $u^{t_k,\rho_k,\theta_k}$ converging in $C^1_{\text{loc}}(O)$ to some $u\in C^{1,\alpha}_{\text{loc}}(O)\cap C^{0,1}(O)$, as $k\rightarrow \infty$. By part $(ii)$ of Lemma \ref{propHt} and the stability of viscosity solutions, $u$ is the unique viscosity solution of \eqref{mainPDE}.  This completes the proof of part $(i)$ of Theorem \ref{mainTheorem}.

\par 5. Now assume in addition that $H$ satisfies \eqref{Hcond2}. According to parts $(iv)$ and $(v)$ of Lemma \ref{propHt},  $H^{t,\rho,\theta}$ also satisfies \eqref{Hcond2}. By assumption, $D^2H\ge \theta'$ for some $\theta'\in (0,1)$. Lemma \ref{propHt} implies that $D^2H^{t,\rho,\theta}\ge \frac{1}{2}\theta'$, for $t>0$ small enough.  The convergence assertion in parts $1-3$ of the current proof implies the estimate
$$
|u^{t,\rho,\theta}|_{C^{1,1}(O')}\le C\left(O',\theta',\max_{|p|\le C}|D^2H^{t,\rho,\theta}(p)| \right)
$$
for $t,\rho$, $\theta>0$ and some universal constant $C$.  Noting that part $(iv)$ of Lemma \ref{propHt} has that $D^2H^{t,\rho,\theta}$ is a.e. locally bounded above, independently of all $t,\rho$, $\theta$ positive and small, we are able to conclude as in part 4 of this proof. 
\end{proof}

\par As a final remark, we mention equations with general boundary conditions 
\begin{equation}\label{genBound}
\begin{cases}
\max\{Lu-f,H(Du)\}=0, \; & x\in O\\
\hspace{1.46in} u=g, \; & x\in \partial O
\end{cases},
\end{equation}
can be handled similarly to PDE \eqref{mainPDE}.  The method presented in this paper works with little alteration provided there exists a subsolution $\underline{u}\in C^1(\overline{O})$ of \eqref{genBound} such that $ \underline{u}|_{\partial \Omega}=g$ and 
$$
H(D\underline{u}(x))<0, \quad x\in O.
$$

\noindent  {\bf Acknowledgements}:  The author would like to thank his PhD dissertation advisor Professor L.C. Evans for his guidance with this work.  Some of this work was completed 
while the author was visiting the home of Theodore P. Hill and Erika Rogers; the author expresses his deepest thanks to Ted and Erika for their hospitality.

\appendix

\newpage


\end{document}